\documentclass[11pt]{article}

\usepackage{graphicx}
\graphicspath{{Figures/}}

\usepackage[square, numbers, comma, sort&compress]{natbib}
\usepackage{graphicx}
\usepackage{tikz}

\usepackage{amsfonts}
\usepackage{amsmath}
\usepackage{amsthm} 
\usepackage{amssymb}
\usepackage{epsfig, graphics}
\usepackage{hyperref}
\usepackage{url}
\usepackage{color}
\usepackage{setspace}
\usepackage{float}
\usepackage{subfig}
\usepackage{colortbl}
\usepackage[margin=1cm]{caption}

\usepackage[square, numbers, comma, sort&compress]{natbib}
\usepackage{graphicx}
\usepackage{tikz}

\hypersetup{colorlinks,linkcolor={blue},citecolor={red},urlcolor={blue}}

\newtheorem{theorem}{Theorem}[section]
\newtheorem{remark}[theorem]{Remark}

\newtheorem{condition}[theorem]{Condition}

\newtheorem{proposition}[theorem]{Proposition}
\newtheorem{lemma}[theorem]{Lemma}

\def\EE{\mathbb{E}}
\def\PP{\mathbb{P}}
\def\NN{\mathbb{N}}

\def\ind{{\rm 1\hspace{-0.90ex}1}}

\newcommand{\vep}{\varepsilon}

\newcommand{\Bin}{\mathsf{Bin}}
\def\tod{\stackrel{d}{\longrightarrow}}
\def\top{\stackrel{p}{\longrightarrow}}

\newcommand{\cU}{\mathcal{U}}
\newcommand{\cE}{\mathcal{E}}
\newcommand{\cB}{\mathcal{B}}

\newcommand{\cD}{\mathcal{D}}

\newcommand{\cA}{\mathcal{A}}
\newcommand{\cG}{\mathcal{G}}

\newcommand\bd{\mathbf{d}}

\def\tod{\stackrel{d}{\longrightarrow}}
\def\top{\stackrel{p}{\longrightarrow}}
 \def\red{\color{red}}
 \def\blue{\color{blue}}

\begin{document}

\title{A Central Limit Theorem for  Diffusion in Sparse Random Graphs}

\author{Hamed Amini\thanks{University of Florida, Department of Industrial and Systems Engineering, Gainesville, FL 32611, USA, email: {\tt aminil@ufl.edu}.} \ \ \ \ \ \ \ \
Erhan Bayraktar\thanks{Department of Mathematics, University of Michigan, 
email: {\tt erhan@umich.edu}}  \ \ \ \ \ \ \ \
Suman Chakraborty \thanks{ Department of Mathematics and Computer Science, Eindhoven University of Technology,
email: {\tt s.chakraborty1@tue.nl}}
}
\maketitle
\abstract{We consider bootstrap percolation and diffusion in sparse random graphs with fixed degrees, constructed by configuration model. Every vertex has two states: it is either active or inactive. We assume that to each vertex is assigned a  nonnegative (integer) threshold.  The diffusion process is initiated by a subset of vertices with threshold zero which consists of initially activated vertices, whereas every other vertex is inactive. Subsequently, in each round, if an inactive vertex with threshold $\theta$ has at least $\theta$ of its neighbours activated, then it also becomes active and remains so forever. This is repeated until no more vertices become activated. The main result of this paper provides a central limit theorem for the final size of activated vertices. Namely, under suitable assumptions on the degree and threshold distributions, we show that the final size of activated vertices has asymptotically Gaussian fluctuations. 
\bigskip


}

\bigskip

\section{Introduction}\label{sec:201pm15my22}
Threshold models of contagion have been used to describe many complex phenomena in diverse areas including epidemiology~\cite{newman02, amini2020epidemic, Pastor-Satorras15, morris00, watts2011simple}, neuronal activity~\cite{tlusty2009remarks, amini-nn}, viral marketing~\cite{Kleinberg07, leskovec2007dynamics, kempe03} and spread of defaults in economic networks~\cite{acm16, amini2019dynamic}.

Consider a connected graph $G=(V,E)$ with $n$ vertices $V=[n]:=\{1,2, \dots, n\}$. Given two vertices $i,j \in [n]$, we write $i \sim j$ if $\{i,j\} \in E$. Every vertex has two states: it is either active or inactive (sometimes also referred to as infected or uninfected). We assume that to each vertex is assigned a  nonnegative (integer) threshold. Let $\theta_i$ be the threshold of vertex $i$. The diffusion is then initiated by the (deterministic) subset of vertices with threshold zero which consists of initially activated vertices 
$\cA_0:=\{i\in V: \theta_i=0\}.$ Subsequently, in each round, if an inactive vertex with threshold $\theta$ has at least $\theta$ of its neighbours activated, then it also becomes active and remains so forever. Namely, at time $t\in \NN$, the (deterministic) set of active vertices is given by
\begin{equation}
\cA_{t}=\{i\in V: \sum_{j: j\sim i} \ind{\{j\in \cA_{t-1}\}} \geq \theta_i\},
\end{equation}
where $\ind{\{\cE\}}$  denotes the indicator of an event $\cE$, i.e., this is 1 if $\cE$ holds and 0 otherwise.

We study above threshold driven contagion model in sparse random graphs with fixed degrees, constructed by configuration model. The interest in this random graph model stems from its ability to mimic some of the empirical properties of real networks, while allowing for tractability (see e.g.,~\cite{Molloy95, hofstad16, durrett07}). We describe this random graph model next.

\subsection{Configuration model}

For each integer $n\in \mathbb{N}$, we consider a system of $n$ vertices $[n] :=\{1,2, \dots,n\}$ endowed with a sequence of initial non-negative integer thresholds $(\theta_{n,i})_{i=1}^n$. We are also given a sequence $\bd_n = (d_{n,i})_{i=1}^n$ of nonnegative integers
$d_{n,1},\ldots,d_{n,n}$ such that $\sum_{i=1}^n d_{n,i}$ is even. By means of the configuration model we define a random multigraph with given degree sequence, denoted by $\cG(n,\bd_n)$. It starts with $n$ vertices and $d_{n,i}$ half-edges corresponding to the $i$-th vertex. Then at each step two half-edges are selected uniformly at random, and a full edge is completed by joining them. The multigraph is constructed when there is no more half edges left. Although self-loops and multiple edges may occur, these become rare (under our regularity conditions below) as $n \to \infty$. It is easy to see conditional on the multigraph being simple graph, we obtain a uniformly distributed random graph with these given degree sequences; see e.g.~\cite{hofstad16}.

Let $u_n(d, \theta)$ be the numer of vertices with degree $d$ and threshold $\theta$:
\begin{equation*}
u_n(d, \theta) := \#\{ i \in [n]:  d_{n,i} = d, \ \theta_{n,i} = \theta \}.
\end{equation*}

We assume the following regularity conditions on the degree sequence and thresholds. 


\begin{condition}\label{cond}
For each $n\in \NN$, $\bd_n = (d_{n,i})_{i=1}^n$ and $(\theta_{n,i})_{i=1}^n$ are sequence of non-negative integers such that
$\sum_{i=1}^n d_{n,i}$ is even and, for some probability distribution $p:\NN^2\to [0,1]$ independent of $n$, with $u_n(d, \theta)$ as defined above:
\begin{enumerate}
    \item[$(C_1)$] as $n\to \infty$, 
    \begin{equation*}
    \frac{u_n(d,\theta)}{n}:=\frac{\#\{ i \in [n]:  d_{n,i} = d, \ \theta_{n,i} = \theta \}}{n} \longrightarrow p(d,\theta)
    \end{equation*}
for every non-negative integers $d$ and $\theta$. We further assume $\sum_{\theta=0}^\infty p(0,\theta)<1$.
\item[$(C_2)$] for every $A>1$, we have
\begin{equation*}
    \sum_{d= 0}^\infty \sum_{\theta=0}^\infty u_n(d,\theta)A^d = \sum_{i=1}^n A^{d_{n,i}} = O(n).
\end{equation*}
\end{enumerate}
\end{condition}

\begin{remark}
Let $(D_n,\Theta_n)$ be random variables with joint distribution $\PP(D_n=d, \Theta_n=\theta)=u_n(d, \theta)/n,$
which is the distribution of dergee and threshold of a random vertex in $\cG(n,\bd_n)$. Moreover, let $(D,\Theta)$ be random variables (over nonnegative integers) with joint distribution $\PP(D=d, \Theta=\theta)=p(d,\theta)$. Let 
\begin{equation*}
\lambda:=\EE[D]= \sum_{d= 0}^\infty \sum_{\theta=0}^\infty dp(d,\theta).
\end{equation*}
Then Condition~\ref{cond} can be rewritten as $(D_n, \Theta_n) \tod (D, \Theta)$ as $n\to \infty$ and $\EE[A^{D_n}]=O(1)$ for each $A>1$, which in particular implies the uniform integrability of $D_n$, so (as $n\to \infty$) 
\begin{equation*}
\frac{\sum_{i=1}^n d_{n,i}}{n} = \EE[D_n] \longrightarrow \EE[D]=\lambda \in (0,\infty).
\end{equation*}
Similarly, all higher moments converge.
\end{remark}

\subsection{Diffusion process in $\cG(n,\bd_n)$}
The aim of this section is to write the activation process described above as a diffusion process. 
We start with the graph $\cG(n,\bd_n)$, and then remove (here the removal of a vertex is the same as activation described above) the initially activated vertex with threshold zero, say $\cA_0$. Now the degree of each vertices in the graph induced by $[n]\setminus \cA_0$ is less than or equal to the degree of that vertex in $\cG(n,\bd_n)$. We denote the degree of vertex $i$ at time $t$ in the evolving graph by $d_{i}^B(t)$ (here the time $t$ is non-negative integer valued, later we will introduce a continuous time). In this process, we remove the vertex $i$ at time $t$ if $d_i^B(t) \leq d_i-\theta_i$ (if there are more than one such $i$, then we we select one arbitrarily). Note that at the end of this procedure all the vertices that are removed will be the set of active vertices, and the rest will remain inactive. 

Let us describe the process by assigning a type $A$ (active) or $B$ (inactive) to each of these vertices. To begin with (at time $t=0$), the initially activated vertices (with threshold zero) are said to be of type $A$, and all vertices not in the initially activated vertices are of type $B$. At time $t>0$, a vertex is said to be of type $A$ if $d_i^B(t) \leq d_i-\theta_i$, otherwise, we call it of type $B$. In particular at the beginning of the diffusion process, all vertices in initially activated vertices are of type $A$, and all vertices not in the initially activated vertices are by definition of type $B$. As we proceed with the algorithm, $d_i^B(t)$ might decrease and a type $B$ vertices may become of type $A$. 
In terms of edges, a half-edge is of type $A$ or $B$ when its endpoint is. As long as there is any half-edge of type $A$, we choose one such half-edge uniformly at random and remove the edge it belongs to. Note that it may result in changing the other endpoint from $B$ to
$A$ (by decreasing $d_i^B(t)$) and thus create new half-edges of type $A$. When there are no half-edges of type $A$ left, we stop the process and the final set of activated vertices is the set of vertices of type $A$ (which are all isolated). 

The next step is to turn this process into a balls and bins problem. In this step we simultaneously run the activation process described above and construct the configuration model. We call a type $A$ (or $B$) vertex as type $A$ (or $B$) bins and similarly consider the half-edges as balls of type $A$ (or $B$) if its end point is type $A$ (or $B$). At each step, we remove first one random ball from the set of balls in $A$-bins and then another ball without restriction. We stop when there are no non-empty $A$-bins. Therefore we alternately remove an $A$-ball and a random ball. We may just as well say that we first remove a random $A$-ball. We then remove balls in pairs, first a ball without restriction, and then a random $A$-ball, and stop with the ball, leaving no $A$-ball to remove. 

The next step is to run the above deletion process in continuous time. Each ball has an exponentially distributed random lifetime with mean $1$, independent of the other balls. In other words balls die and are removed at rate $1$, independent of each other. Also when a ball dies we remove a random $A$ ball (in terms of configuration model these two half-edges form an edge). We stop when we should remove an $A$ ball but there is no such a ball. Let $H_A(t)$ and $H_B(t)$ denote the number of $A$ balls and $B$ balls respectively, at time $t$. Note that these parameters depend on $n$, but we will omit the subscript as it is clear from the context. We also let $A_n(t)$ and $B_n(t)$ be the number of $A$ bins and $B$ bins at time $t$. 

Let $\tau_n$ be the stopping time of the above diffusion process. Note that there are no $A$ balls left at $\tau_n$. However we pretend that we delete a (nonexistent) $A$ ball at $\tau_n$-th step and denote $H_A(\tau) = -1$. Therefore the stopping time $\tau_n$ may be characterized by $H_A(\tau) = -1$, and $H_A(t) \geq 0$ for all $0\leq t < \tau_n$. Denoting by $\cA_n^*$ the final set of activated vertices, we observe that $|\cA^*_n|= n-B_n(\tau_n).$ 
Next we consider the balls only. The total number of balls at time $t$ is $H_A(t)+H_B(t)$. In the evolution of this process each ball dies with rate $1$ and another ball is removed upon its death. Therefore, $H_A(t)+H_B(t)$ is a death process with rate $2$.  

\subsection{Main theorems}
The main result of this paper provides a central limit theorem for the final size of activated vertices. Namely, when the degree and threshold sequences satisfy Condition~\ref{cond}, we show that the final size of activated vertices $|\cA^*_n|$ has asymptotically Gaussian fluctuations. 

Let
\begin{align*}
b(d,z,\ell):=& \PP(\Bin(d,z)= \ell) = \binom{d}{\ell} z^\ell (1-z)^{d-\ell}, \\
\beta(d,z,\ell):=& \PP(\Bin(d,z)\geq \ell) = \sum_{r=\ell}^d \binom{d}{r} z^r (1-z)^{d-r},
\end{align*}
and $\Bin(d,z)$ denotes the binomial distribution with parameters $d$ and $z$.

For the following theorems, we will use the notation
\begin{align*}
&\widehat{a}_n(t):= 1-{\frac{1}{n}} \sum_{d=0}^{\infty}\sum_{\theta=1}^d  u_n(d,\theta) \beta(d,e^{-t},d-\theta+1),
\end{align*}
where $u_n(d,\theta)$  is the number of vertices with degree $d$ and threshold $\theta$, and 
\begin{align*}
&\widehat{a}(t):= 1- \sum_{d=0}^{\infty}\sum_{\theta=1}^d  p(d,\theta) \beta(d,e^{-t},d-\theta+1).
\end{align*}

\begin{theorem}\label{thm_n-main1}
Assume that Condition~\ref{cond} holds. Let $\tau_n' \leq \tau_n$ be a stopping time such that $\tau_n' \top t_0$ for some $t_0 \geq 0$. Then in $\cD[0,\infty)$,  as $n\to \infty$,
\begin{align*}
&n^{-1/2}\bigl(A_n(t \wedge \tau_n')-n\widehat{a}_n(t \wedge \tau_n')\bigr) \tod Z_{A}(t \wedge t_0),
\end{align*}
for all $t\leq t_0$, where $Z_{A}$ is a continuous Gaussian process on $[0,t_0]$ with mean $0$ and variance
\begin{equation*}\label{eqn:sigma2}
\begin{aligned}
     \sigma_A^2(t):= \sum_{d=1}^{\infty}&\sum_{\theta=1}^{d} e^{-(d-\theta+1)t}\\
     &\times
     \left(\Delta_{d, \theta, d-\theta+1}(t)
    +\sum_{\ell=d-\theta+2}^{\infty}(d-\theta+1){\ell-1\choose d-\theta+1 }\int_{0}^t(e^{-s}-e^{-t})^{\ell-(d-\theta+2)}e^{-s} \Delta_{d,\theta, \ell}(s) \, ds\right),
    \end{aligned}
\end{equation*}
where 
\begin{equation*}
\Delta_{d,\theta,\ell}(t)=p(d,\theta)\left(1-e^{2\ell t}\beta(d,e^{-t},\ell) +  2\ell  \int_{0}^t e^{2\ell s}\beta(d,e^{-s},\ell) \, ds  \right).
\end{equation*} 

\end{theorem}

\medskip

The proof of above theorem is provided in Section~\ref{sec:proofMain}. 

\medskip

Let us define 
\begin{align*}
h_B(z) :=\sum_{d=1}^{\infty}\sum_{\theta=1}^{d} p(d,\theta) \sum_{\ell=d-\theta+1}^d \ell b(d,z,\ell) \ \ \text{and} \ \ h_A(z):=\lambda  z^2 - h_B(z),
\end{align*}
which also gives us the plausible candidate for the limit of our stopping time
\begin{equation*}
\widehat{z} :=\sup\{z\in [0,1]:  h_A(z)=0 \}.
\end{equation*}
Note that the solution $\widehat{z}$ always exists since $h_A(0)=0$. Also, $h_A(z)$ is continuous in $z$. We will further assume that if $\widehat{z}\neq 0$, then it is not a local minimum of $h_A(z)$, i.e., $\alpha:=h'_A(\widehat{z})>0$.

\begin{lemma}\label{lem:stop}
Let $\tau_n$ be the stopping time of the diffusion process such that $H_A(\tau_n) = -1$ for the first time.
Then $\tau_n \top -\ln \widehat{z}$ as $n\rightarrow \infty$, where $\top$ denotes the convergence in probability.
\end{lemma}

The proof of lemma is provided in Section~\ref{sec:proofLem}. 

\medskip

We are now interested in the final number of activated vertices (bins) which is given by $\lim_{t \rightarrow \infty}A_n(t \wedge \tau_n)$. Note that from Theorem~\ref{thm_n-main1} with the stopping time $\tau_n$ (as in Lemma~\ref{lem:stop}), we have for $t\geq -\ln{\widehat{z}}$,
\begin{align*}
&n^{-1/2}\left(A_n(t \wedge \tau_n)-n\widehat{a}_n(t \wedge \tau_n)\right) \rightarrow Z_{A}(t \wedge -\ln{\widehat{z}}).
\end{align*}
Therefore for $t\geq \ln{\widehat{z}}$,
\begin{align*}
&n^{-1/2}\left(A_n(t \wedge \tau_n)-n\widehat{a}_n(t \wedge \tau_n)\right) \tod Z_{A}(-\ln{\widehat{z}}).
\end{align*}

The following theorem provides a central limit theorem for the final size of activated vertices. Let us denote by
$$
h_A^n(z) :=  \frac{\sum_{i=1}^n d_{n,i}}{n}z^2-\frac{1}{n}\sum_{d=1}^{\infty}\sum_{\theta=1}^{\infty}\sum_{\ell=d-\theta+1}^{d}\ell u_n(d,\theta) b(d,z,\ell),
$$
and let $\widehat{z}_n$ be the largest $z\in[0,1]$ such that $h_A^n(z)=0$.

\begin{theorem}\label{prop:645pm17may22}
Assume Condition~\ref{cond}. Let $\widehat{\tau}=-\ln{\widehat{z}}$ and $\widehat{\tau}_n=-\ln{\widehat{z}_n}$. We have:
\begin{itemize}
\item If $\widehat{z}=0$ then asymptotically almost all nodes become active and $|\cA^*_n|=n-o_p(n)$.
\item If $\widehat{z}\neq0$ and $\widehat{z}$ is a stable solution, i.e. $\alpha:=h'_A(\widehat{z})>0$, then 
\begin{equation*}
n^{-1/2}\left(|\cA^*_n|-  n\widehat{a}_n(\widehat{\tau}_n)\right) \tod Z_{\text{Act}},
\end{equation*}
where $Z_{\text{Act}}= \frac{{\widehat{a}'}(\widehat{\tau})}{\alpha} Z_{HA}(\widehat{\tau})+ Z_A(\widehat{\tau})$, where $(Z_{HA}(\widehat{\tau}),  Z_A(\widehat{\tau}))$ jointly follow a Gaussian distribution that is described in Proposition \ref{prop:446pm17may22}.  

\end{itemize}

\end{theorem} 

The proof of above theorem is given in Section~\ref{proof:prop:645pm17may22}, where we first derive a joint functional central limit theorem for the processes $(A_n(t),H_A(t))$, from which we derive the theorem.

\subsection{Relation to bootstrap percolation and $k$-core}
We now discuss our results with respect to related literature. 

The diffusion model we consider in this paper can be seen as a generalization of bootstrap percolation and $k$-core in any graph $G=(V,E)$. 

In the case of equal thresholds ($\theta_i=\theta$ over all vertices), our model is equivalent to bootstrap percolation (with deterministic initial activation). This process was introduced by Chalupa, Leath and Reich~\cite{CLR79} in 1979 as a simplified model of some magnetic disordered systems. A short survey regarding applications of bootstrap percolation processes can be found in~\cite{Adl03}. Recently, bootstrap percolation has been studied on varieties of random graphs models, see e.g., \cite{janson2012bootstrap} for random graph $G_{n,p}$, \cite{fountoulakis2018phase, amini2014bootstrap, aminibootstrapIRG} for inhomogeneous random graphs and \cite{balpit07, amini10, lelarge12b} for the configuration model; see also~\cite{liu2012core, liu2016control, newman2006structure, boccaletti2006complex, dorogovtsev2008critical}.

The $k$-core of a graph $G$ is the largest induced subgraph of $G$ with minimum vertex degree at least $k$. The $k$-core of an arbitrary finite
graph can be found by removing vertices of degree less than $k$, in an arbitrary order, until no such vertices exist. By setting the threshold of vertex $i$ as $\theta_i=(d_i-k+1)_+=\max\{d_i-k+1,0\}$, we find that $B_n(\tau_n)$ will be the size of $k$-core in the random graph $G(n,\bd_n)$. 

The questions concerning the existence, size and structure of the $k$-core in random graphs, have attracted a lot of attention over the last few decades, see e.g., \cite{pittel1996sudden, luczak1991size} for random graph $G_{n,p}$, \cite{bayraktar2020k, riordan2008k} for inhomogeneous random graphs and  \cite{janson2007simple, molloy2005cores, cooper2004cores, janson2008asymptotic, fernholz2004cores} for the configuration model.

In particular, more closely related to our paper, \cite{janson2008asymptotic}  analyze the asymptotic normality of the $k$-core  for sparse random graph $G_{n,p}$ and for configuration model.  We continue on the same line as \cite{janson2008asymptotic} and generalize partly their results by allowing different threshold levels to each of vertices. Our proof technique is also inspired by \cite{janson2008asymptotic}. In particular, we look at the spread of activation (or infection) and constructing the configuration model simultaneously. Then we express the number of inactive vertices at a particular time point in terms of a martingale. After that we appeal to a martingale limit theorem from \cite{jacod2013limit} to derive the limiting distribution.

\paragraph{Notation.} We let $\NN$ be the set of nonnegative integers. 
Let $\{ X_n \}_{n \in \NN}$ be a sequence of real-valued random variables on a probability space
$ (\Omega, \mathbb{P})$. 
If $c \in \mathbb{R}$ is a constant, we write $X_n \stackrel{p}{\longrightarrow} c$ to denote that $X_n$ converges in probability to $c$.
That is, for any $\epsilon >0$, we have $\mathbb{P} (|X_n - c|>\epsilon) \rightarrow 0$ as $n \rightarrow \infty$.
We write $X_n = o_{p} (a_n)$, if $|X_n|/a_n$ converges to 0 in probability. We use $\tod$ for convergence in distribution.
If $\mathcal{E}_n$ is a measurable subset of $\Omega$, for any $n \in \NN$, we say that the sequence
$\{ \mathcal{E}_n \}_{n \in \mathbb{U}}$ occurs with high probability ({\bf w.h.p.}) if $\mathbb{P} (\mathcal{E}_n) = 1-o(1)$, as
$n\rightarrow \infty$.
Also, we denote by 
$\Bin (k,p)$ a binomial  distribution  corresponding to the number of
successes of a sequence of $k$ independent Bernoulli trials each having probability of success  $p$.
We will suppress the dependence of parameters on the size of the network $n$, if it is clear from the context. We use the notation $\ind{\{\cE\}}$ for the indicator of an event $\cE$ which is 1 if $\cE$ holds and 0 otherwise. We let $\cD[0,\infty)$ be the standard space of right-continuous functions with left limits on $[0,\infty)$ equipped with the Skorohod topology (see e.g.~\cite{jacod2013limit, kallenberg1997foundations})

\section{Preliminaries}

In this section we provide some preliminary lemmas that will be used in our proofs. 

\subsection{Some death process lemmas}
Consider a pure death process with rate 1. This process starts with some number of balls whose lifetimes are i.i.d. rate 1 exponentials. 
\begin{lemma}[Death Process Lemma] \label{lem_n-death}
Let $N_n(t)$ be the number of balls alive at time $t$ in a rate $1$ death process with $N_n(0)=n$. Then
\[
\sup_{t\geq 0}|N_n(t)/n - e^{-t}| \top 0 \text{  as  }n\rightarrow \infty. 
\]
\end{lemma}
\begin{proof}
 $1-N_n(t)/n$ is the empirical distribution function of the $n$ lifetimes, which are i.i.d. random variables with the distribution function $1-e^{-t}$. Therefore the result follows using Glivenko-Cantelli theorem (see e.g. \cite[Proposition 4.24]{kallenberg1997foundations}).
\end{proof}
\begin{lemma}[Number of Balls Centrality Lemma]
The number of balls $H_A(t)+H_B(t)$ follow a pure death process, and  
\begin{align*}
\sup_{0\leq t\leq \tau_n}|H_A(t)+H_B(t)-n\lambda e^{-2t}| = o_p(n).
\end{align*}
\end{lemma}
\begin{proof} In the evolution of this process each ball dies with rate $1$ and another ball is removed upon its death. Therefore $A(t)+B(t)$ is a death process with rate $2$. Therefore the lemma follows using Lemma \ref{lem_n-death}.
\end{proof}

\subsection{Martingale limit theorems}
We recall some martingale theory that are going to be useful in proving Theorem~\ref{thm_n-main1}. Let $X$ be a martingale defined on $[0,\infty)$, we denote its quadratic variation of $X$ by $[X,X]_t$, and the bilinear extension of quadratic variation to two martingales $X$ and $Y$ by $[X,Y]_t$. If $X$ and $Y$ be two martingales with path-wise finite variation, then 
\begin{equation}\label{eqn:monoct7230pm}
[X,Y]_t :=\sum_{0<s\leq t}{ \Delta X(s)\Delta Y(s)},
\end{equation}
where $\Delta X(s):= X(s) -X(s-)$ is the jump of $X$ at $s$. Similarly, $\Delta Y(s):= Y(s) -Y(s-)$. In our context there will only be countable number of jumps for the martingales under consideration, and  the sum in Equation~\eqref{eqn:monoct7230pm} will be finite. We will assume $[X,Y]_0 = 0$. For vector-valued martingales $X = (X_i)_{i=1}^m$ and $Y = (Y_i)_{i=1}^n$, we define the square bracket $[X,Y]$ to be the matrix $([X_i,Y_j])_{i,j}$. A real-valued martingale $X(s)$ on $[0, t]$ is an $L^2$ if and only if
$\mathbb{E}[X,X]_t <\infty$ and $\mathbb{E}|X(0)|^2 <\infty$, and then $\mathbb{E}|X(t)|^2 = \mathbb{E}[X,X]_t+\mathbb{E}|X(0)|^2$. We will use the following martingale limit theorem from~\cite{jacod2013limit}, see also~\cite[Proposition 4.1]{janson2008asymptotic}.
\begin{proposition}\label{prop:martingale}
For each $n\geq 1$, let $M_{n}(t)=(M_{ni}(t))_{i=0}^q$ be a $q$-dimensional martingale on $[0,\infty)$ and $M_n(0)=0$. Also $\Sigma(t)$ be a continuous positive semi-definite function such that for every fixed $t\geq 0$
\begin{align*}
[M_n,M_n]_t \top \Sigma(t)  \ \ \text{as} \ \ n\to \infty, \\
\sup_{n}\EE[M_{ni},M_{ni}]_t<\infty, \ \ i=1,\dots, q.
\end{align*}
Then $M_n\tod M$, in $\cD[0,\infty)$ where $M$ is continuous $q$-dimensional Gaussian martingale with $\mathbb{E} M(t)=0$ and covariances 
$\mathrm{Cov}(M(t))= \Sigma(t)$.
\end{proposition}

In the next section, we will apply Proposition~\ref{prop:martingale} to stopped processes.

\section{Proofs}\label{sec:proof}
In this section we present the proofs of Theorem~\ref{thm_n-main1}, Lemma~\ref{lem:stop} and Theorem~\ref{proof:prop:645pm17may22}.
\begin{remark}
In the proof of Theorem~\ref{thm_n-main1}, we always consider the processes up to time $\tau_n'\leq \tau_n$. Although, sometimes it will be possible to extend the process (for example by removing non-existent balls), that will not be relevant for us, and thus we will always stop the process at $\tau_n'$.
\end{remark}
\subsection{Proof of Theorem~\ref{thm_n-main1}}\label{sec:proofMain}
We denote the number of (alive) balls at time $t$ by $W_n(t)$. Clearly $W_n(t)= H_A(t) +H_B(t)$. Let $m_n:=\frac{1}{2} \sum_{i=1}^n d_{n,i}$ denotes the total number of edges in $\cG(n,\bd_n)$. In our construction $W_n(0)=2m_n-1$, and $W_n$ decreases by $2$ each time a ball dies. The death happens with rate $1$, and therefore $W_n(t)+\int_{0}^t 2W_n(s)\, ds$ is a martingale on $[0,\tau_n']$. In differential form 
\begin{equation}\label{eqn:monoct7351pm}
\, dW_n(t) = -2W_n(t)\, dt +\, d\mathcal{M}(t),
\end{equation}
where $\mathcal{M}$ is a martingale. Now by Ito's lemma and \eqref{eqn:monoct7351pm},
\[
\, d(e^{2t} W_n(t)) = e^{2t} \, dW_n(t) +2e^{2t} W_n(t) \, dt = e^{2t} \, d\mathcal{M}(t), 
\]
which implies that $\widehat{W}_n(t) : = e^{2t} W_n(t)$ is another martingale. Note that distinct balls die at distinct time with probability one, also all jumps in $W_n(t)$ equals $-2$. The quadratic variation of $\widehat{W}_n(t)$ is given by
\begin{align}\label{eqn:monoct7736pm}
[\widehat{W}_n, \widehat{W}_n]_t &= \sum_{0<s\leq t} {|\Delta\widehat{W}_n(s)|}^2 = \sum_{0<s\leq t} \left(e^{2s}|\Delta{W}_n(s)|\right)^2 \notag \\
& = \sum_{0<s \leq t} e^{4s} ({W}_n(s-)- W_n(s)) ({W}_n(s-)- W_n(s)) \notag \\
&= \sum_{0<s \leq t} 2e^{4s} ({W}_n(s-)- W_n(s)) = \int_{0}^{t} 2 e^{4s} \, d(-W_n(s)) \notag \\
&= -2e^{4t} W_n(t) +2W_n(0) +\int_{0}^t 8 e^{4s}W_n(s) \, ds.
\end{align}
Using the fact that $W_n(t)$ is a decreasing function in $t$ and $W_n(0) = 2m_n-1$ we get
\begin{align*}
[\widehat{W}_n, \widehat{W}_n]_t \leq 2e^{4t} \int_{0}^t \,d(-W_n(s)) \leq 4 m_n e^{4t}.
\end{align*}

Let the stopped martingale on $[0,\infty)$ be $W_n^*(t) := \frac{1}{n}\widehat{W}(t \wedge \tau_n')$. Then $W_n^*$ is a martingale on $[0,\infty)$ and, for every $T<\infty$, the quadratic variation of this martingale is
\begin{equation}\label{eqn:frifeb21325pm}
[W_n^*, W_n^*]_T = \frac{1}{n^2} [W_n, W_n]_{T \wedge \tau_n'} \leq \frac{4e^{4T}m
_n}{n^2} \longrightarrow 0,
\end{equation}
as $n$ goes to infinity. Also $0\leq W_n^*(t) \leq 2m_n/n = O(1)$. Therefore by Proposition \ref{prop:martingale} on $\cD[0,\infty)$, 
$W_n^*(t) - W_n^*(0) \top 0$ uniformly on $[0,T]$, and as $n\to \infty$,
\begin{equation}\label{eqn:oct151728}
n^{-1} \sup_{0\leq t\leq T \wedge \tau_n'} |W_n(t)- W_n(0)e^{-2t}| \longrightarrow 0.
\end{equation}

Now using \eqref{eqn:monoct7736pm}, \eqref{eqn:frifeb21325pm} and \eqref{eqn:oct151728}  we get for every $t \in [0, T\wedge \tau_n']$,
\begin{align}\label{eqn:monoct7746pm}
[\widehat{W}_n, \widehat{W}_n]_t &= -2e^{4t}W_n(0) e^{-2t} +2W_n(0)+ \int_{0}^t 8 e^{4s}W_n(0) e^{-2s} \, ds +o_p(n) \notag\\
&=2W_n(0)(e^{2t}-1) +o_p(n) = 2(2m_n-1)(e^{2t}-1) +o_p(n) \notag\\
&= 2\lambda n (e^{2t}-1) +o_p(n).
\end{align}

Note that in the last display, \eqref{eqn:oct151728} ensures that we can use the approximation $W_n(t)= W_n(0)e^{-2t}+o_p(n)$ on $[0, T\wedge \tau_n']$.  Let $W_i(t)$ be the number of balls for vertex (bin) $i$ and denote by (for $d,\theta, \ell\in \NN$)
$$\mathcal{B}_{d, \theta,\ell}(t) = \{i \in [n]: d_i= d, \theta_i=\theta,  W_i(t)= \ell\},$$ 
the set of vertices (bins) with degree $d$, threshold $\theta$ and $\ell$ balls at time $t$.  Let $B_{d, \theta,\ell}(t)=|\mathcal{B}_{d, \theta,\ell}(t)|$. Therefore, the total number of (inactive) $B$ bins 
at time $t$ is given by
\begin{equation*}
B_n(t)= \sum_{d=0}^{\infty}\sum_{\theta=1}^{\infty}\sum_{ \ell= d-\theta+1}^d B_{d, \theta,\ell}(t)=\sum_{d=0}^{\infty}\sum_{\theta=1}^{\infty} \widetilde{B}_{d,\theta, d-\theta+1}(t), 
\end{equation*}
where $\widetilde{B}_{d,\theta,\ell} = \sum_{r=\ell}^d B_{d, \theta,r}(t)$ denotes the number of (inactive) $B$ bins  at time $t$ with initial degree $d$ and threshold $\theta$ with at least $\ell$ balls. In the rest of the proof we will derive a central limit theorem for the quantity $B_n(t)$, which is the number of inactive vertices at time $t$. This will immediately give us our desired result since $A_n(t)+B_n(t) =n$.

Note that $\widetilde{B}_{d,\theta,\ell}$ decreases by one only when a ball dies in an uninfected (inactive) bin with initially $d$ balls and threshold $\theta$ that has exactly $\ell$ balls, and there are precisely $\ell B_{d,\theta,\ell}$ many such balls, therefore
\begin{equation}\label{eqn:wedoct9538pm}
\, d\widetilde{B}_{d,\theta,\ell}(t) = -\ell B_{d,\theta,\ell}(t)\, dt +\, d\mathcal{M}(t),
\end{equation}
where $\mathcal{M}$ is a martingale.

Let us now define the following quantity 
\begin{equation}\label{frifeb21156pm}
\widehat{B}_{d,\theta,\ell}(t) := e^{\ell t} \widetilde{B}_{d,\theta,\ell}(t),
\end{equation}
which gives
\begin{align*}
\,d\widehat{B}_{d,\theta,\ell}(t) =\ell e^{\ell t} \widetilde{B}_{d,\theta,\ell}(t)\, dt + e^{\ell t} \,d\widetilde{B}_{d,\theta,\ell}(t).
\end{align*}

Plugging in \eqref{eqn:wedoct9538pm} we get
\begin{align*}
\,d\widehat{B}_{d,\theta,\ell}(t) &=\ell e^{\ell t} \widetilde{B}_{d,\theta,\ell}(t) \, dt -\ell e^{\ell t}B_{d,\theta,\ell}(t)\, dt +e^{\ell t} \, d\mathcal{M}(t) \\
&= \ell e^{\ell t} \widetilde{B}_{d,\theta,\ell+1}(t) \, dt +e^{\ell t} \, d\mathcal{M}(t) \\
&= \ell e^{-t}\widehat{B}_{d,\theta,\ell+1}(t)\, dt + \, d\mathcal{M'}(t),
\end{align*}
where $\, d\mathcal{M'}(t) = e^{\ell t} \,d \mathcal{M}(t)$. Since this is yet another martingale differential, we can define the following martingale for every fixed $d\geq \ell\geq 0$
\begin{equation}\label{eqn:frifeb21216pm}
M_{d,\theta,\ell}(t) := \widehat{B}_{d,\theta,\ell}(t) - \ell\int_{0}^t e^{-s}\widehat{B}_{d,\theta,\ell+1}(s) \, ds.
\end{equation}
The quadratic variation will be same  as that of $\widehat{B}_{d,\theta,\ell}(t)$, i.e.,
\begin{align}\label{eqn:satfeb22124pm}
[M_{d,\theta,\ell},M_{d,\theta,\ell}]_t &= \sum_{0<s\leq t}|\Delta M_{d,\theta,\ell}(s)|^2 = \sum_{0<s\leq t}|\Delta \widehat{B}_{d,\theta,\ell}(s)|^2 \notag \\
&=\sum_{0<s\leq t} e^{2\ell s}|\Delta \widetilde{B}_{d,\theta,\ell}(s)|^2  \notag \\
&=  \int_{0}^t e^{2\ell s} \, d(- \widetilde{B}_{d,\theta,\ell}(s)).
\end{align}

Let us now define the centered version of $M_{d,\theta,\ell}$ as follows
\begin{equation}\label{eqn:1245pm25jan21}
\widetilde{M}_{d,\theta,\ell}(t) := n^{-1/2}(M_{d,\theta,\ell}(t) -{M}_{d,\theta,\ell}(0)).
\end{equation}
This is of course the centered martingale where for $\ell \leq d$,
\begin{equation*}
M_{d,\theta,\ell}(0) = \widetilde{B}_{d,\theta,\ell}(0) =  \sum_{r=\ell}^{\infty} \sum_{i=1}^{n}  \ind\{i \in \mathcal{B}_{d,\theta,r}(0)\} = \sum_{i=1}^{n}  \ind\{i \in \mathcal{B}_{d,\theta,d}(0)\}  .
\end{equation*}
The quadratic variation of $\widetilde{M}_{d,\theta,\ell}(t)$ can be calculated using integration by parts as follows
\begin{align*}
[\widetilde{M}_{d,\theta,\ell}, \widetilde{M}_{d,\theta,\ell}]_t &= \frac{1}{n}[M_{d,\theta,\ell},M_{d,\theta,\ell}]_t \notag \\
&= \frac{1}{n}\left(\widetilde{B}_{d,\theta,\ell}(0) - e^{2\ell t}\widetilde{B}_{d,\theta,\ell}(t) + 2\ell  \int_{0}^t e^{2\ell s}\widetilde{B}_{d,\theta,\ell}(s) \, ds\right).
\end{align*}

Recall from \eqref{frifeb21156pm} that  $\widehat{B}_{d,\theta,\ell}(t) := e^{\ell t} \widetilde{B}_{d,\theta,\ell}(t),$ 
where $\widetilde{B}_{d,\theta,\ell}(t)$ is the number of (uninfected) $B$ bins with at least $\ell$ balls at time $t$ with initial number of balls (degree) $d$ and threshold $\theta$. Also, recall that balls die independently with rate $1$. Let us denote $$\cU_n(d,\theta) = \{i\in [n]: d_{n,i}=d, \theta_{n,i}=\theta\},$$ so that $u_n(d,\theta) = |\cU_n(d,\theta)|$. 
This gives also $u_{n}(d,\theta) = \sum_{\ell}B_{d,\theta,\ell}(0)$. 

Recall that $B_{d, \theta, \ell}(t) = |\cB_{d,\theta,\ell}(t)|$ where $\cB_{d,\theta,\ell}(t) = \{i\in [n]: d_i=d, \theta_i=\theta, W_i(t)=\ell\}$. Since each ball dies with rate one independent of each other and survival probability of a ball after time $t$ is equal to $e^{-t}$. A bin from $\cU_n(d,\theta)$ has at least $\ell$ balls at time $t$ with probability
\begin{equation*}
\beta(d,e^{-t},\ell) = \sum_{r=\ell}^d{d \choose r} (e^{-t})^r (1-e^{-t})^{d-r}. 
\end{equation*}

Hence we get
\begin{equation*}
\mathbb{E}[\widehat{B}_{d,\theta,\ell}(t)] = e^{\ell t}  \mathbb{E}[\widetilde{B}_{d,\theta,\ell}(t)] = e^{\ell t} \sum_{v \in \cU_{n}(d,\theta)}\mathbb{P}(\text{out of  }d \text{  at least  } \ell \text{  balls survive} ) ,
\end{equation*}
which gives
\begin{equation*}
\widehat{b}_{d,\theta,\ell}(t)=\mathbb{E}[\widehat{B}_{d,\theta,\ell}(t)] = e^{\ell t}  u_{n}(d,\theta) \beta(d,e^{-t},\ell).
\end{equation*}

Now note that (for $\theta\geq 1$)
\begin{equation*}
\widetilde{B}_{d,\theta,\ell}(0) = u_n(d,\theta) \ \  \text{   and   } \ \ \mathbb{E}\bigl[\widetilde{B}_{d,\theta,\ell}(t)\bigr] = u_n(d,\theta) \beta(d,e^{-t},\ell).
\end{equation*}

Hence, by using Condition~\ref{cond} and Glivenko-Cantelli's lemma (since each bins are independent), we get 
\begin{equation}
\label{eqn:612pm16may22}
\sup_{0\leq t\leq \tau_n'}\left|\frac{1}{n}\widetilde{B}_{d,\theta,\ell}(t)-p(d,\theta)\beta(d,e^{-t},\ell)\right|\top 0,
\end{equation}
as $n\to \infty$. Therefore, the quadratic variation of $\widetilde{M}_{d,\theta,\ell}(t)$ satisfies (for $\theta\geq 1$ and $\ell \leq d$) (see also Equation~\ref{eqn:644pm07jan21})
 \begin{align*}
   [\widetilde{M}_{d,\theta,\ell}, \widetilde{M}_{d,\theta,\ell}]_t :=& p(d,\theta)\left(1-e^{2\ell t}\beta(d,e^{-t},\ell) +  2\ell  \int_{0}^t e^{2\ell s}\beta(d,e^{-s},\ell) \, ds  \right)+o_p(1)\notag\\
    =& \Delta_{d,\theta,\ell}(t) +o_p(1).
\end{align*} 

We will also use the following crude estimate. Using \eqref{eqn:satfeb22124pm} we get that
\begin{equation}\label{eqn:aug24942pm}
[\widetilde{M}_{d,\theta,\ell}, \widetilde{M}_{d,\theta,\ell}]_t = \frac{1}{n} [M_{d,\theta,\ell},M_{d,\theta,\ell}]_t \leq \frac{1}{n} e^{2\ell t} \widetilde{B}_{d,\theta,\ell}(0)= \frac{u_n(d,\theta)}{n} e^{2\ell t}.
\end{equation}

Therefore we can apply Proposition \ref{prop:martingale} to the stopped process at $\tau_n'$
\begin{equation}\label{eqn:satfeb22135pm}
\widetilde{M}_{d,\theta,\ell}(t \wedge \tau_n') \tod Z_{d,\theta,\ell}(t \wedge t_0) \text{  in  }D[0,\infty),
\end{equation}
where $Z_{d,\theta,\ell}$ is a Gaussian process with $\mathbb{E}Z_{d,\theta,\ell}(t) =0$ and covariance $\Delta_{d,\theta,\ell}(t)$.

Also note that $\widetilde{B}_{d,\theta,\ell}(t)$ and $\widetilde{B}_{d',\theta',\ell'}(t)$ can not change together almost surely for $(d,\theta,\ell)\neq (d',\theta',\ell')$, therefore $[\widetilde{M}_{{d,\theta,\ell}}, \widetilde{M}_{{d', \theta', \ell'}}] = 0$. Thus for a finite set $S$, $\bigl(\widetilde{M}_{d,\theta,\ell}\bigr)_{(d,\theta,\ell)\in S}$ converges to $(Z_{d,\theta,\ell})_{(d,\theta,\ell)\in S}$ in distribution, and $(Z_{d,\theta,\ell})$ are independent. 

Let us now express $\widehat{B}_{d,\theta,\ell}$ in terms of ${M}_{d,\theta,\ell}$ so that we can apply the limit theorems for $\widetilde{M}_{d,\theta,\ell}$'s to get limit theorems for $\widehat{B}_{d,\theta,\ell}$.

Using \eqref{eqn:frifeb21216pm} one can write 
\begin{align}\label{eqn:satfeb22740pm}
\widehat{B}_{d,\theta,\ell}(t) = M_{d,\theta,\ell}(t)+\sum_{r=\ell+1}^{d} \ell {r-1\choose \ell}\int_{0}^t(e^{-s}-e^{-t})^{r-\ell-1}e^{-s} M_{d, \theta,r}(s) \, ds.
\end{align}
To see how to get \eqref{eqn:satfeb22740pm} from \eqref{eqn:frifeb21216pm}, note that 
\begin{equation*}
\begin{aligned}
   \widehat{B}_{d,\theta,\ell}(t) &:=  M_{d,\theta,\ell}(t) + \ell\int_{0}^t e^{-s}\widehat{B}_{d,\theta,\ell+1}(s) \, ds \\
    &= M_{d,\theta,\ell}(t) + \ell\int_{0}^t e^{-s} M_{d,\theta,\ell+1}(s) \, ds + (\ell+1)\ell\int_{0<s_1<s<t}^t e^{-s} e^{-s_1}\widehat{B}_{d,\theta,\ell+2}(s_1) \, ds_1 \, ds.
    \end{aligned}
\end{equation*}
Proceeding like this by repeatedly using \eqref{eqn:frifeb21216pm}, we can obtain \eqref{eqn:satfeb22740pm}.
Since $M_{d,\theta,\ell}$ is a martingale $\mathbb{E}[M_{d,\theta,\ell}(t)]= M_{d,\theta,\ell}(0)$, and using \eqref{eqn:satfeb22740pm} we get $ \mathbb{E}[\widehat{B}_{d,\theta,\ell}(t)]=\widehat{b}_{d,\theta,\ell}(t)$, where we define for $t\geq 0$,
\begin{equation}\label{eqnsatfeb22735pm}
\widehat{b}_{d,\theta,\ell}(t) = M_{d,\theta,\ell}(0)+\sum_{r=\ell+1}^{d}\ell{r-1\choose \ell }\int_{0}^t(e^{-s}-e^{-t})^{r-\ell-1}e^{-s} M_{d,\theta, \ell}(0) \, ds.
\end{equation}
Using \eqref{eqn:satfeb22740pm}, \eqref{eqnsatfeb22735pm} and \eqref{eqn:1245pm25jan21}, it turns out that the centered version of $\widehat{B}_{d,\theta, \ell}(t)$ satisfies 
\begin{align*}
&\frac{1}{\sqrt{n}}\left(\widehat{B}_{d,\theta,\ell}(t) - \widehat{b}_{d,\theta,\ell}(t)\right) \notag \\
&= \frac{1}{\sqrt{n}} \left({M}_{d,\theta,\ell}(t)-{M}_{d,\theta,\ell}(0)+\sum_{r=\ell+1}^{\infty}\ell{r-1\choose \ell }\int_{0}^t\bigl(e^{-s}-e^{-t}\bigr)^{r-\ell-1}e^{-s} \bigl({M}_{d,\theta,\ell}(s)-M_{d,\theta,\ell}(0)\bigr) \, ds \right)\notag\\
 &   = \widetilde{M}_{d,\theta,\ell}(t)
    +\sum_{r=\ell+1}^{\infty}\ell{r-1\choose \ell}\int_{0}^t(e^{-s}-e^{-t})^{r-\ell-1}e^{-s} \widetilde{M}_{d, \theta, \ell}(s) \, ds. 
\end{align*}
Now using \eqref{eqn:aug24942pm}  we get (for $\ell \leq d$ and for any $T\geq 0$):
\begin{equation*}
   \mathbb{E} \left[\widetilde{M}_{d,\theta,\ell},\widetilde{M}_{d,\theta,\ell}\right]_T \leq e^{2\ell T} \frac{u_n(d,\theta)}{n} \leq \frac{1}{n} e^{2\ell T}\sum_{d\geq \ell} \sum_{\theta} u_n(d,\theta).
\end{equation*}
Then using Condition~\ref{cond} and the simple fact that for $A>1$,
$$A^\ell\sum_{d\geq \ell} \sum_{\theta} u_n(d,\theta) \leq \sum_{d\geq \ell} \sum_{\theta} u_n(d,\theta) A^d \leq  \sum_{d} \sum_{\theta} u_n(d,\theta)A^d,$$ we get for any $T\geq 0$, by setting $A= e^{2T+2}$, there is a constant $C_A$ such that
\begin{equation*}
   \mathbb{E} \left[\widetilde{M}_{d,\theta,\ell},\widetilde{M}_{d,\theta,\ell}\right]_T \leq e^{2\ell T}  C_A A^{-\ell}  = C_A e^{-2\ell}.
\end{equation*}
 Now using Doob's $L^2-$inequality we get
 $$
 \mathbb{E} \left( \sup_{0\leq t\leq T} \widetilde{M}_{d,\theta,\ell}^2(t)\right) \leq 4 \mathbb{E}\left[\widetilde{M}_{d,\theta,\ell},\widetilde{M}_{d,\theta,\ell}\right]_T \leq C' e^{-2\ell}.
 $$
Therefore using Cauchy-Schwarz inequality,
\begin{equation}\label{eqn:aug26540pm}
   \mathbb{E} \left( \sup_{0\leq t\leq T}| \widetilde{M}_{d,\theta,\ell}(t)|\right)  \leq C'e^{-\ell}.
\end{equation}
Therefore by \eqref{eqn:satfeb22135pm} and Fatou's lemma we get 
\begin{equation}\label{eqn:220pm25jan21}
    \mathbb{E}\left(\sup_{0\leq t\leq t_0}\left|Z_{d,\theta,\ell}(t)\right|\right) \leq C'e^{-\ell}.
\end{equation}

 Let us define 
 \begin{equation*}
 R_{d,\theta,\ell}(t):= \sum_{r=\ell+1}^{\infty}\ell{r-1\choose \ell}\int_{0}^t(e^{-s}-e^{-t})^{r-\ell-1}e^{-s} \widetilde{M}_{d,\theta,r}(s) \, ds.
 \end{equation*}
 Then we have (since $(e^{-s}-e^{-t})^{r-\ell-1}\leq (1-e^{-t})^{r-\ell-1}$ for $s\in[0,t]$)
 \begin{equation*}
\mathbb{E}\left[ \sup_{0\leq t\leq T}\left|R_{d,\theta,\ell}(t)\right|\right] \leq \sum_{r=\ell+1}^{\infty}\ell{r-1\choose \ell}\int_{0}^T(1-e^{-t})^{r-\ell-1}e^{-s}\mathbb{E} \left[\sup_{0\leq t\leq T} \left|\widetilde{M}_{d,\theta,r}(s)\right|\right] \, ds.
    \end{equation*}

Therefore \eqref{eqn:aug26540pm} yields 
\begin{equation}\label{eqn:aug26603pm}
\mathbb{E}\left( \sup_{0\leq t\leq T}\left|R_{d,\theta,\ell}(t)\right|\right) \leq C'  \sum_{r=\ell+1}^{\infty}\ell e^{-r} {r-1\choose \ell} (1-e^{-T})^{r-\ell}.
\end{equation}

Clearly, for fixed $T$, \eqref{eqn:aug26603pm} converges to zero uniformly in $n$, as $\ell \rightarrow \infty$. Now using \eqref{eqn:satfeb22740pm}, \eqref{eqnsatfeb22735pm}, \eqref{eqn:aug26603pm} and applying \cite[Theorem 4.2]{billingsley2013convergence} we get
\begin{equation*}
     \frac{1}{\sqrt{n}}\left(\widehat{B}_{d,\theta,\ell}(t \wedge \tau_n') - \widehat{b}_{d,\theta,\ell}(t \wedge \tau_n')\right) \tod \widetilde{Z}_{d,\theta,\ell}(t \wedge t_0),
    \end{equation*}
    in $\cD[0,\infty)$ for each $\ell$, where   
    \begin{equation}\label{eqn:221pm25jan21}
        \widetilde{Z}_{d,\theta,\ell}(t) := Z_{d,\theta,\ell}(t)
    +\sum_{r=\ell+1}^{\infty}\ell{r-1\choose \ell}\int_{0}^t(e^{-s}-e^{-t})^{r-\ell-1}e^{-s} Z_{d,\theta,r}(s) \, ds. 
    \end{equation}

 Again using \eqref{eqn:220pm25jan21} we obtain that the sum in \eqref{eqn:221pm25jan21} almost surely uniformly converges for $t\leq t_0$, and this gives $\widetilde{Z}_{d,\theta,\ell}$ is continuous for each $(d,\theta,\ell)$. Using \eqref{frifeb21156pm}, we can write 
 $$
 \frac{e^{-\ell t}}{\sqrt{n}}\left(\widehat{B}_{d,\theta,\ell}(t ) - \widehat{b}_{d,\theta,\ell}(t )\right) :=  \frac{1}{\sqrt{n}}\left(\widetilde{B}_{d,\theta,\ell}(t) - e^{-\ell t}\widehat{b}_{d,\theta,\ell}(t )\right) .
 $$

 Now recall that we are interested in the case $\ell=d-\theta+1$. More particularly, we will derive a central limit theorem of the following quantity 
    
    \begin{align}\label{eqn:106pm08feb21}
    &\frac{1}{\sqrt{n}} \sum_{d=1}^{\infty}\sum_{\theta=1}^{d}\left(\widetilde{B}_{d,\theta,d-\theta+1}(t) - e^{-(d-\theta+1) t}\widehat{b}_{d,\theta,d-\theta+1}(t )\right)  \notag \\
    &= \frac{1}{\sqrt{n}} \sum_{d=1}^{\infty}\sum_{\theta=1}^{d}e^{-(d-\theta+1) t}\left(\widehat{B}_{d,\theta, d-\theta+1}(t ) - \widehat{b}_{d,\theta,d-\theta+1}(t )\right)\notag \\
    &=  \sum_{d=1}^{\infty}\sum_{\theta=1}^{d} e^{-(d-\theta+1) t}\widetilde{M}_{d, \theta, d-\theta+1}(t) \notag\\
    &+ \sum_{d=1}^{\infty}\sum_{\theta=1}^{d}e^{-(d-\theta+1) t}\sum_{r=d-\theta+2}^{\infty}(d-\theta+1){r-1\choose d-\theta+1 }\int_{0}^t(e^{-s}-e^{-t})^{r-(d-\theta+1)-1}e^{-s} \widetilde{M}_{d,\theta,r}(s) \, ds. 
    \end{align}
Note that from \eqref{eqn:satfeb22135pm} we know the limiting distribution of $\widetilde{M}_{d, \theta, d-\theta}(t)$, and therefore, it will be sufficient to show that the contribution from the tail of \eqref{eqn:106pm08feb21} is negligible. To be precise, let us define two terms (we ignore the contribution from $e^{-(d-\theta) t}$, which will always be bounded above by $1$)
 $$\widetilde{R}_{\ell}(t):= \sum_{d =\ell}^{\infty}\sum_{\theta=1}^{d} \widetilde{M}_{d, \theta, d-\theta+1}(t),$$ and  
 
 $$\widehat{R}_{\ell}(t):= \sum_{d= \ell}^{\infty}\sum_{\theta=1}^{d}\sum_{r=d-\theta+2}^{d}(d-\theta+1){r-1\choose d-\theta+1 }\int_{0}^t(e^{-s}-e^{-t})^{r-(d-\theta+1)-1}e^{-s} \widetilde{M}_{d, \theta, r}(s) \, ds.$$
In the following lemma we show that $\mathbb{E} \left[ \sup_{0\leq t\leq T} \left| \widetilde{R}_{\ell}(t)\right|\right]$, and $ \mathbb{E} \left[ \sup_{0\leq t\leq T} \left| \widehat{R}_{\ell}(t)\right|\right]$ converges to zero as $\ell\rightarrow \infty$ uniformly in $n$.

\begin{lemma}\label{lem:418pm27feb21}
We have 
$\mathbb{E} \left[ \sup_{0\leq t\leq T} \left| \widetilde{R}_{\ell}(t)\right|\right]\to 0$ and $ \mathbb{E} \left[ \sup_{0\leq t\leq T} \left| \widehat{R}_{\ell}(t)\right|\right] \to 0$, as $\ell\rightarrow \infty$, uniformly in $n$.
\end{lemma}

\begin{proof}
From Condition~\ref{cond} and the fact that for all $A>1$,
$A^d u_n(d,\theta) \leq \sum_{d \geq \ell}u_n(d,\theta) A^d$ we find that for any $A>1$ there is a constant $C_A$ such that 
\begin{equation*}
   \mathbb{E} \left[\widetilde{M}_{d,\theta,\ell},\widetilde{M}_{d,\theta,\ell}\right]_T \leq e^{2\ell T} \frac{u_n(d,\theta)}{n}.
\end{equation*}

Using Doob's $L^2-$inequality we get
 \begin{equation}\label{eqn:126pm08feb21}
 \mathbb{E} \left[ \sup_{0\leq t\leq T} \widetilde{M}_{d,\theta,\ell}^2(t)\right] \leq 4 \mathbb{E}\left[\widetilde{M}_{d,\theta,\ell},\widetilde{M}_{d,\theta,\ell}\right]_T \leq 4e^{2rT} \frac{u_n(d,\theta)}{n}.
  \end{equation}
    Therefore,
    \begin{align*}
        \mathbb{E} \left[ \sup_{0\leq t\leq T} \left| \widetilde{R}_{\ell}(t)\right|\right] \leq \sum_{d= \ell}^\infty\sum_{\theta=1}^{d} \mathbb{E} \left( \sup_{0\leq t\leq T} \left|\widetilde{M}_{d,\theta,d-\theta}(t)\right|\right).
    \end{align*}
    Using Cauchy-Schwarz inequality, and \eqref{eqn:126pm08feb21}
     \begin{align*}
        \mathbb{E} \left[ \sup_{0\leq t\leq T} \left| \widetilde{R}_{N,n}(t)\right|\right] \leq 2\sum_{d=\ell}^{\infty}\sum_{\theta=1}^{d} e^{(d-\theta)T} \sqrt{\frac{u_n(d,\theta)}{n}} \leq \frac{2}{\sqrt{n}} \sum_{d\geq \ell}e^{dT}\sum_{\theta=1}^{d}  \sqrt{u_n(d,\theta)}. 
    \end{align*}

         Using Condition~\ref{cond}, for all $A>1$,
$A^d u_n(d,\theta) =O(n) $. Therefore for any $A>1$ there is a constant $C_A$ such that
\begin{equation*}
    \mathbb{E} \left[ \sup_{0\leq t\leq T} \left| \widetilde{R}_{\ell}(t)\right|\right] \leq C_A \sum_{d\geq \ell} de^{dT}A^{-d/2}.
\end{equation*}
In particular, choosing $A=e^{4T}$, we see that $\mathbb{E} \left[ \sup_{0\leq t\leq T} \left| \widetilde{R}_{\ell}(t)\right|\right]$ converges to zero as $\ell\rightarrow \infty$ uniformly in $n$.

    Now let us define 
    $$
    \widehat{R}_{\ell}(t):= \sum_{d= \ell}^{\infty}\sum_{\theta=1}^{d}\sum_{r=d-\theta+1}^{d}(d-\theta){r-1\choose d-\theta }\int_{0}^t(e^{-s}-e^{-t})^{r-(d-\theta)-1}e^{-s} \widetilde{M}_{d, \theta, r}(s) \, ds.
    $$ 
    Observe that the index $r$ is at most $d$, and therefore using Cauchy-Schwarz inequality, and \eqref{eqn:126pm08feb21}
\begin{equation}\label{eqn:424pm08feb21}
\mathbb{E}\left[ \sup_{0\leq t\leq T}\left|\widetilde{M}_{d, \theta, r}(s)\right|\right]\leq 2e^{sT}\sqrt{\frac{u_n(d,\theta)}{n}}. 
\end{equation}

We can now compute the following tail bound
\begin{equation*}
 \mathbb{E} \left( \sup_{0\leq t\leq T} \left| \widehat{R}_{\ell}(t)\right|\right) \leq \sum_{d= \ell}^{\infty}\sum_{\theta=1}^{d}\sum_{r=d-\theta+1}^{d}(d-\theta){r-1\choose d-\theta }(1-e^{-t})^{r-(d-\theta)} 2e^{sT}\sqrt{\frac{u_n(d,\theta)}{n}}
\end{equation*}
Again we have $u_n(d,\theta)\leq C_A A^{-d}n$ for some constant $C_A>0$. Plugging in $A=e^{8T}$ we get
\begin{equation}\label{eqn:315pm08feb21}
 \mathbb{E} \left[ \sup_{0\leq t\leq T} \left| \widehat{R}_{\ell}(t)\right|\right] \leq C_A\sum_{d=\ell}^{\infty}e^{-3dT}\sum_{\theta=1}^{d}\sum_{r=d-\theta+1}^{\infty}(d-\theta){r-1\choose d-\theta }(1-e^{-T})^{r-(d-\theta)}.
\end{equation}
We can write 
\begin{eqnarray}\label{eqn:314pm08feb21}
    \sum_{r=d-\theta+1}^{\infty}(d-\theta){r-1\choose d-\theta }(1-e^{-t})^{r-(d-\theta)} &=& \sum_{r=0}^{\infty} r {r +d-\theta-1\choose r}(1-e^{-T})^{r} \notag\\
    &=& (d-\theta)e^{(d-\theta+1)T}(1-e^{-T}),
\end{eqnarray}
where in the second equality we used the expectation of negative binomial distribution. Now plugging this in \eqref{eqn:315pm08feb21} we have
\begin{equation*}\label{eqn:315pm08feb21ee}
 \mathbb{E} \left[ \sup_{0\leq t\leq T} \left| \widehat{R}_{\ell}(t)\right|\right] \leq C_A\sum_{d\geq \ell}d^2 e^{-2dT},
\end{equation*}
which again goes to zero uniformly in $n$ as $\ell\rightarrow \infty$. This completes the proof of Lemma~\ref{lem:418pm27feb21}.
\end{proof}


By using Lemma~\ref{lem:418pm27feb21}, \eqref{eqn:106pm08feb21} and \cite[Theorem 4.2]{billingsley2013convergence} we get
\begin{equation}\label{eqn:240pm16may22}
\begin{aligned}
        \frac{1}{\sqrt{n}} \sum_{d=1}^{\infty}\sum_{\theta=1}^{d}&\left(\widetilde{B}_{d,\theta,d-\theta+1}(t\wedge \tau_n') - e^{-(d-\theta+1) (t \wedge \tau_n')}\widehat{b}_{d,\theta,d-\theta+1}(t\wedge \tau_n' )\right) \\
        &\tod   \sum_{d=1}^{\infty}\sum_{\theta=1}^{d}e^{-(d-\theta+1) (t \wedge \tau_n')}\widetilde{Z}_{d,\theta,d-\theta+1}(t \wedge t_0),
\end{aligned}    
\end{equation}
as $n\to \infty$, that is 
\begin{equation}
\label{eqn:903pm16may22}
\begin{aligned}
    \frac{1}{\sqrt{n}}&\left(B_n(t\wedge \tau_n')- \sum_{d=1}^{\infty}\sum_{\theta=1}^{d} u_n(d,\theta) \beta(d,e^{-t},d-\theta+1)\right) \\
    &\tod 
    \sum_{d=1}^{\infty}\sum_{\theta=1}^{d}e^{-(d-\theta+1) (t \wedge \tau_n')}\widetilde{Z}_{d,\theta,d-\theta+1}(t \wedge t_0)=:Z_A(t\wedge t_0).
\end{aligned}    
\end{equation}
We are now done except for showing that $Z_A(t)$ is continuous. Using \eqref{eqn:424pm08feb21} we get
\begin{equation*}
  \mathbb{E}\left[\sup_{0\leq t\leq T} \frac{1}{\sqrt{n}} \left|\widehat{B}_{d,\theta,\ell}(t) - \widehat{b}_{d,\theta,\ell}(t)\right|\right] \leq 2e^{sT}\sqrt{\frac{u_n(d,\theta)}{n}}
    +2\sum_{r=\ell+1}^{\infty}\ell {r-1\choose \ell}(1-e^{-t})^{r-\ell} e^{sT}\sqrt{\frac{u_n(d,\theta)}{n}}. 
\end{equation*}

Again using $u_n(d,\theta)\leq C_A A^{-d}n$ and writing the second term in terms of negative binomial distribution we get
\begin{equation*}
     \mathbb{E}\left(\sup_{0\leq t\leq T} \frac{1}{\sqrt{n}}  \left|\widehat{B}_{d,\theta,\ell}(t) - \widehat{b}_{d,\theta,\ell}(t)\right|\right) \leq C_A e^{dT}A^{-d/2}+ C_A e^{dT}A^{-d/2}e^{\ell T+T}(1-e^{-T})\ell.
\end{equation*}
Since $\ell\leq d$, by choosing $A=e^{4T+2}$, we get 
\begin{equation}\label{eqn:457pm08feb21}
    \mathbb{E}\left[\sup_{0\leq t\leq T} \frac{1}{\sqrt{n}} \left|\widehat{B}_{d,\theta,\ell}(t) - \widehat{b}_{d,\theta,\ell}(t)\right|\right] \leq C_A de^{-d}.
\end{equation}
Using \eqref{eqn:457pm08feb21} and Fatou's lemma we get
\begin{equation*}
    \mathbb{E}\left[\sup_{0\leq t\leq t_0}\left|\widetilde{Z}_{d,\theta,d-\theta}(t)\right|\right] \leq C_A de^{-d}.
\end{equation*}    
This in turn implies that $\sum_{d=1}^{\infty}\sum_{\theta=1}^{d} e^{-(d-\theta) t}\widetilde{Z}_{d,\theta,d-\theta}(t)$, almost surely converges uniformly in $t\leq t_0$ and therefore $\sum_{d=1}^{\infty}\sum_{\theta=1}^{d}\widetilde{Z}_{d,\theta,d-\theta}(t)$ is continuous almost surely. 
Finally, we finish the proof by proving the tail bound in \eqref{eqn:106pm08feb21}. 


This completes the proof of Theorem~\ref{thm_n-main1}.
\subsection{The stopping time}\label{sec:stop}
Note that we can write
\begin{equation*}
    H_B(t) =  \sum_{d=1}^{\infty}\sum_{\theta=1}^{d}\sum_{\ell=d-\theta +1}^d \ell B_{d,\theta,\ell}(t)=  \sum_{d=1}^{\infty}\sum_{\theta=1}^{d}\sum_{\ell= d-\theta+1}^{d} \ell \sum_{v \in [n]} \ind\{v \in \cB_{d,\theta,\ell}(t)\}. 
\end{equation*}
It is precisely the number of balls that were not initially infected (active), and remain non-infected at time $t$. First let us consider the term $\sum_{\ell= d-\theta+1}^{d} \ell \sum_{v \in [n]} \ind\{v \in \cB_{d,\theta,\ell}(t)\}$.
Consider those $u_n(d,\theta)$ bins with degree $d$ and threshold $\theta \geq 1$ (not initially infected). For $k=1,2\ldots,u_n(d,\theta)$, let $T_k$ be the time $\ell$-th ball is removed from the $k$-the such bin. Then 
\begin{equation*}
    \left|\{k:T_k \leq t\}\right| = \sum_{r=0}^{d-\ell} \sum_{v \in [n]}\ind\{v \in \cB_{d,\theta,r}(t)\}.
\end{equation*} 
For the $k$-th bin we get
\begin{equation*}
    \mathbb{P}\left(T_k \leq t\right) = \sum_{r=0}^{d-\ell}{d \choose r} \left(e^{-t}\right)^{r}\left(1-e^{-t}\right)^{d-r} :=\sum_{r=0}^{d-\ell}b(d,e^{-t},r),
\end{equation*} 
 out of the $u_n(d,\theta)$ bins. Since all bins are independent of each other, using Glivenko-Cantelli lemma we get that
\begin{equation*}
 \sup_{t\geq 0} \left| \frac{1}{n} \sum_{r=0}^{d-\ell}\sum_{v \in [n]} \ind\{v \in \cB_{d,\theta,r}(t)\} -\frac{u_n(d,\theta)}{n}\sum_{r=0}^{d-\ell}b(d,e^{-t},r) \right|\longrightarrow 0,
\end{equation*}
in probability as $n \rightarrow \infty$. 
By Condition~\ref{cond}, $u_n(d,\theta) = p(d,\theta)n+o(n)$, therefore we get 
\begin{equation}\label{eqn:644pm07jan21}
 \sup_{t\geq 0} \left| \frac{1}{n} \sum_{r=0}^{d-\ell}\sum_{v \in [n]}\ind\{v \in \cB_{d,\theta,r}(t)\} - p(d,\theta)\sum_{r=0}^{d-\ell}b(d,e^{-t},r) \right|\longrightarrow 0,
\end{equation}
in probability as $n \rightarrow \infty$. Since \eqref{eqn:644pm07jan21} is true for all $0\leq \ell \leq d$, taking difference of two consecutive terms we get for each $0\leq r\leq d$,
\begin{equation*}
     \sup_{t\geq 0} \left| \frac{1}{n} \sum_{v \in [n]}\ind\{v \in \cB_{d,\theta,r}(t)\} -p(d,\theta)b(d,e^{-t},r) \right|\longrightarrow 0,
\end{equation*}
and therefore taking a sum over $r$, and using  Condition~\ref{cond},
\begin{equation*}
     \sup_{t\geq 0} \left| \frac{1}{n} \widetilde{B}_{d,\theta,\ell}(t) -p(d,\theta)\beta(d,e^{-t},\ell) \right|\top  0.
\end{equation*}

Also combining Condition~\ref{cond} and \eqref{eqn:644pm07jan21}  we get
\begin{equation}\label{eqn:712pm07jan21}
   \sup_{t \geq 0} \left| \frac{H_B(t)}{n}  -  \sum_{d=1}^{\infty}\sum_{\theta=1}^{d}\sum_{\ell= d-\theta+1}^{d} \ell p(d,\theta)b(d,e^{-t},\ell)\right| \top 0,
\end{equation}
in probability as $n \rightarrow \infty$.
Recall that the stopping time $\tau$ was defined as $H_A(\tau) = -1$. Also note that from \eqref{eqn:712pm07jan21} and \eqref{eqn:oct151728} we get that
\begin{equation}\label{eqn:oct151738}
  \sup_{t \geq 0}\left|  \frac{H_A(t)}{n} - \lambda e^{-2t} + \sum_{d=1}^{\infty}\sum_{\theta=1}^{d}\sum_{\ell= d-\theta+1}^{d} \ell  p(d,\theta)b(d,e^{-t},\ell)\right| \top 0,
\end{equation}
in probability as $n \rightarrow \infty$.
We can then  write 
$h_B(z) =\sum_{d=1}^{\infty}\sum_{\theta=1}^{d}\sum_{\ell= d-\theta+1}^{d} \ell p(d,\theta)b(d,z,\ell)$ and $h_A(z)=\lambda  z^2 - h_B(z)$, as the limit of $\frac{H_B(t)}{n}$ and $\frac{H_A(t)}{n}$, by setting $z=e^{-t}$. Also \eqref{eqn:oct151738}, gives us the plausible candidate for the limit of our stopping time
$
\widehat{z} :=\sup\{z\in [0,1]: h_A(z)=0 \}.
$
We further assume that if $\widehat{z}$ is a non-zero then it is not a local minimum of $h_A(z)$.

\subsection{Proof of Lemma~\ref{lem:stop}}\label{sec:proofLem}
To show this let us take a constant $t_1>0$ such that $t_1 < -\ln \widehat{z}$. This means $\widehat{z}<1$, and hence $\lambda  > h_B(1)$. Therefore using the fact that $\lambda  z^2 - h_B(z)$ is continuous we get $h_B(z) < \lambda z^2$ for $z \in (\widehat{z},1]$. This again gives $h_B(e^{-t}) < \lambda e^{-2t}$ for all $t \leq t_1$. Since $[0,t_1]$ is compact again using the continuity of $h_B(z)$ we get $h_B(e^{-t}) - \lambda e^{-2t}<-c$ for some $c>0$. Therefore on the set $\tau< t_1$ we will have $h_B(e^{-\tau}) - \lambda e^{-2\tau}<-c$.  On the other hand, since $H_A(\tau) =-1$, as $n\rightarrow \infty$ in \eqref{eqn:oct151738} ${H_A(\tau)}/{n} \rightarrow 0$, which gives a contradiction, and thus
$\mathbb{P}(\tau \leq t_1) \rightarrow 0$ as $n \rightarrow \infty$. Now let us choose $t_2 < \tau$ where $t_2 \in ( -\ln{\widehat{z}}, -\ln{(\widehat{z} -\vep)})$. By our assumption $\widehat{z}$ is not a local minimum of $h_A(z)=\lambda  z^2 - h_B(z)$, therefore there is an $\vep>0$ such that $\lambda  e^{-2t_2} - h_B(e^{-t_2})=-c$ for some $c>0$. Now by definition if $\tau>t_2$, then $H_B(t_2)\geq 0$. Plugging these in \eqref{eqn:oct151738} we get $\frac{H_B(t_2)}{n} - \lambda e^{-2t_2} + h_B(t_2) \geq c$. This gives $\mathbb{P}(\tau \geq t_2) \rightarrow 0$ as $n \rightarrow \infty$. Since $t_1$ and $t_2$ can be arbitrary close to $-\ln \widehat{z}$, the proof of the claim is complete.

\subsection{Proof of Theorem~\ref{prop:645pm17may22}}\label{proof:prop:645pm17may22}
We first derive a joint functional central limit theorem for the processes $(A_n(t),H_A(t))$, from which we prove the theorem.

\begin{proposition}
\label{prop:446pm17may22}
Let $\tau_n' \leq \tau_n$ be a stopping time such that $\tau_n' \top t_0$ for some $t_0 \geq 0$. Then in $\cD[0,\infty)$,  as $n\to \infty$,
\begin{equation}
\label{eqn:444pm17may}
  n^{-1/2}  \left(A_n(t\wedge \tau_n')- n\widehat{a}_n(t),H_A(t\wedge \tau_n')- n h_A^n(e^{-
  t})\right) \tod (Z_A(t\wedge t_0), Z_{HA}(t\wedge t_0) ), 
\end{equation}
with $Z_A$ as in Theorem~\ref{thm_n-main1}, and $Z_{HA}(t)$ is a Gaussian process that is described in the proof of this proposition. (The covariance could be calculated from \eqref{eqn:903pm16may22} and \eqref{eqn:911pm16may22}.)
\end{proposition}

For the sake of readability, we postpone the proof of the proposition to the end of this section. We continue with the proof of Theorem~\ref{prop:645pm17may22}. Part (i) follows form~\cite{amini10}. Consider now the case when $\widehat{z}\neq0$ and $\widehat{z}$ is a stable solution, i.e. $\alpha:=h'_A(\widehat{z})>0$. Using similar arguments as in the proof of \cite[Lemma 2.3]{janson2008asymptotic}, one can show that  ${\widehat{a}_n}$ converges to ${\widehat{a}}$ and ${h}_A^n$ converges to ${h}_A$, together with their derivatives, i.e., ${\widehat{a}_n}'$ converges to ${\widehat{a}}'$ and ${h_A^n}'$ converges to $h_A'$, uniformly on $[0,1]$.

Hence, since for small enough $\delta>0$, $h_A(\widehat{z}+\delta)>0$, and $h_A(\widehat{z}-\delta)<0$,  $h_A^n$ has a zero at $\widehat{z}_n$ in $(\widehat{z}-\delta, \widehat{z}+\delta)$ for sufficiently large $n$. Further, since $h_A^n \to h_A$ uniformly, we have $h_A^n>0$ in the interval $[\widehat{z}+\delta,1]$. Since $\delta>0$ is arbitrary, we have $\widehat{z}_n \to \widehat{z}$ as $n\to \infty$. Let us write $\widehat{\tau}_n= -\ln{\widehat{z}_n}$, and $\widehat{\tau}= -\ln{\widehat{z}}$ , therefore $\widehat{\tau}_n\to \widehat{\tau}$ as $n\to\infty$. \par

We now use the Skorohod coupling theorem so that we can consider all the random variables in \eqref{eqn:444pm17may} to be defined on the same probability space and the limit holds almost surely. Now we apply Proposition \ref{prop:446pm17may22} with $\tau_n$, and Lemma \ref{lem:stop} gives us that $\tau_n\to \widehat{\tau}=-\ln{\widehat{z}}$. Also, since the limits are continuous almost surely, we can say \eqref{eqn:444pm17may} holds uniformly on $[0,\widehat{\tau}+1]$. Therefore taking $t  = \tau_n$ (this is less than $\widehat{\tau}+1$ for large $n$ almost surely), we get
\begin{equation}
    \begin{aligned}
    H_A(\tau_n)&= n h_A^n(e^{-\tau_n})+ n^{1/2} Z_{HA}(\widehat{\tau}\wedge \tau_n) +o(n^{1/2})\\
    &= n h_A^n(e^{-\tau_n})+ n^{1/2} Z_{HA}(\widehat{\tau}) +o(n^{1/2}),
    \end{aligned}
\end{equation}
where using continuity of $Z_{HA}$, we absorb the error term in $o(n^{1/2})$. Since $H_A(\tau_n) =-1$, therefore 
\begin{equation}
\label{eqn:529pm17may22}
    h_A^n(e^{-\tau_n}) = n^{-1/2}Z_{HA}(\widehat{\tau}) +o(n^{-1/2}).
\end{equation}

Using the Mean-Value theorem, we can write for some $\xi_n \in [\widehat{z}_n,z_n ]$ or $[z_n,\widehat{z}_n ]$,
\begin{equation}
\label{eqn:530pm17may22}
h_A^n(e^{-\tau_n})=h_A^n(e^{-\tau_n}) - h_A^n(e^{-\widehat{\tau}_n}) = h_A^n(z_n) - h_A^n(\widehat{z}_n) = {h_A^n}'(\xi_n)(z_n - \widehat{z}_n).
\end{equation}
We have $z_n\to \widehat{z}$ and $\widehat{z}_n\to \widehat{z}$, and therefore $\xi_n \to \widehat{z}$. Now, from ${h_A^n}' \to {h_A}'$ uniformly, we get ${h_A^n}'(\xi_n)\to {h_A}'(\widehat{z})=\alpha $. Therefore \eqref{eqn:529pm17may22}, and \eqref{eqn:530pm17may22} we get
\begin{equation}
    \label{eqn:532pm17may22}
    z_n-\widehat{z}_n = n^{-1/2}\frac{1}{\alpha} \left(Z_{HA}(\widehat{\tau})+o(1)\right).
\end{equation}
Now, we use the Mean value theorem on $A_n$ for some $\xi_n \to \widehat{z}$, and \eqref{eqn:532pm17may22} to obtain
\begin{equation}
    \label{eqn:552pm17may22}
    \begin{aligned}
    n^{-1/2}A_n(\tau_n)&= n^{1/2}\widehat{a}_n(\tau_n)+ Z_A(\widehat{\tau})+o(1) \\
    &= n^{1/2}\widehat{a}_n(\widehat{\tau})) + n^{1/2}{\widehat{a}_n}'(e^{-\xi_n})(z_n - \widehat{z}_n)+ Z_A(\widehat{\tau})+o(1) \\
    &= n^{1/2}\widehat{a}_n(\widehat{\tau}) + \frac{{\widehat{a}'}(\widehat{\tau})}{\alpha} Z_{HA}(\widehat{\tau})+ Z_A(\widehat{\tau})+o(1).
    \end{aligned}
\end{equation}
In the third equality we have used the fact that ${\widehat{a}_n}' \to {\widehat{a}}'$ uniformly in $t$. The proof is thus complete since $|\cA^*_n|= A_n(\tau_n)$. We end this section by presenting the proof of Proposition~\ref{prop:446pm17may22}.

\begin{proof}[Proof of Proposition~\ref{prop:446pm17may22}]
First note that we can write the number of (uninfected) $B$ balls at time $t$ as
\begin{equation*}
H_B(t)= \sum_{d=0}^{\infty}\sum_{\theta=1}^{\infty}\sum_{ \ell= d-\theta+1}^d \ell B_{d, \theta,\ell}(t)=\sum_{d=0}^{\infty}\sum_{\theta=1}^{\infty}(d-\theta+1)\widetilde{B}_{d,\theta, d-\theta+1} +  \sum_{d=0}^{\infty}\sum_{\theta=1}^{\infty}\sum_{ \ell= d-\theta+2}^d\widetilde{B}_{d,\theta, \ell}.
\end{equation*}
This gives
\begin{equation}
    \label{eqn:336pm16may22}
   H_A(t)= W_n(t)-H_B(t)= W_n(t) -\sum_{d=0}^{\infty}\sum_{\theta=1}^{\infty}(d-\theta+1)\widetilde{B}_{d,\theta, d-\theta+1} -  \sum_{d=0}^{\infty}\sum_{\theta=1}^{\infty}\sum_{ \ell= d-\theta+2}^d\widetilde{B}_{d,\theta, \ell}(t).  
\end{equation}
We wish to prove a joint central limit theorem for $(A_n(t), H_A(t))$. Let us first outline the procedure. 
We have already expressed $\widehat{B}_{d,\theta, \ell}(t)$ (which is the centered and scaled version of $\widetilde{B}_{d,\theta, \ell}(t)$) in terms of $\widetilde{M}_{d,\theta,\ell}(t)$'s in \eqref{eqn:satfeb22740pm}, and therefore $H_A(t)$ is also implicitly a linear combination of $W_n(t)$, and $\widetilde{M}_{d,\theta,\ell}(t)$'s (since $\widetilde{B}_{d,\theta, \ell}(t)$'s are also linear combinations of $\widetilde{M}_{d,\theta,\ell}(t)$). Therefore once we can prove that the joint distribution $(W_n, \widetilde{M}_{d,\theta,\ell})$ is Gaussian, we can derive the joint distribution of $(A_n(t), H_A(t))$ (since both $A_n(t)$, $H_A(t)$ are linear combination of independent random variables which are jointly Gaussian).  To do that, let us define $\widetilde{\widehat{W}}_n(t)= n^{-1/2}\left(\widehat{W}_n(t)-\widehat{W}_n(0)\right) $, and note that using \eqref{eqn:monoct7746pm} we get for every $t \in [0, T\wedge \tau_n]$,
\begin{equation}
    [\widetilde{\widehat{W}}_n, \widetilde{\widehat{W}}_n]_{t}=  2\lambda(e^{2t}-1) +o_p(1).
\end{equation}
Therefore using Proposition \ref{prop:martingale} on the stopped process
\begin{equation*}
    \widetilde{\widehat{W}}_n(t\wedge \tau_n') \tod \widehat{Z}_W(t\wedge t_0),
\end{equation*}
as $n\to \infty$, where $\widehat{Z}_W$ is a continuous Gaussian process with mean zero, and variance $2\lambda(e^{2t}-1)$. Now since $\widehat{W}_n(t) : = e^{2t} W_n(t)$, we get that
\begin{equation}
\label{eqn:630pm16may22}
    n^{-1/2}\left({W}_n(t\wedge \tau_n')-e^{-2t}{W}_n(0)\right) \tod {Z}_W(t\wedge t_0),
\end{equation}
as $n\to \infty$, where $Z_W$ is a continuous Gaussian process with mean zero, and variance $2\lambda(e^{-2t}-e^{-4t})$. Moreover, for $t\in [0, T\wedge \tau_n]$, $\ell \geq d-\theta+1$, and $\theta\geq 1$
\begin{equation}\label{eqn:531pm16may22}
\begin{aligned}
     \left[\widetilde{\widehat{W}}_n, \widetilde{M}_{d,\theta,\ell}\right]_t &=\frac{1}{n}\sum_{0<s\leq t}{ \Delta \widehat{W}_n(s)\Delta {M}_{d,\theta,\ell}(s)} = \frac{1}{n}\sum_{0<s\leq t} \Delta \widehat{W}_n(s)\Delta \widehat{B}_{d,\theta,\ell}(s) \\
    & = \frac{1}{n}\sum_{0<s\leq t} e^{(2+\ell)s} \Delta W_n(s)\Delta \widetilde{B}_{d,\theta,\ell}(s).
\end{aligned}     
\end{equation}
Now when $\widetilde{B}_{d,\theta,\ell}$ jumps by $1$ for some $\ell \geq d-\theta+1$, the $W_n$ jumps by $-2$, therefore using \eqref{eqn:612pm16may22}, \eqref{eqn:531pm16may22} yields
\begin{equation}
    \left[\widetilde{\widehat{W}}_n, \widetilde{M}_{d,\theta,\ell}\right]_t = \frac{1}{n}\int_{0}^t2 e^{(2+\ell)s}  \,d (- \widetilde{B}_{d,\theta,\ell}(s)) = \int_{0}^t2 e^{(2+\ell)s}  \,d (-\beta(d,e^{-s},\ell)) +o_p(1).
\end{equation}
Therefore using Proposition \ref{prop:martingale} we get the joint convergence of $(\widetilde{M}_{d,\theta,\ell}(t\wedge \tau_n'),\widetilde{\widehat{W}}_n(t\wedge \tau_n'))$ for $\ell\geq d-\theta+1$. Thus combining \eqref{eqn:630pm16may22}, the joint convergence of $(\widetilde{M}_{d,\theta,\ell}(t\wedge \tau_n'),\widetilde{\widehat{W}}_n(t\wedge \tau_n'))$ for $\ell\geq d-\theta+1$, and \eqref{eqn:240pm16may22} we get   
\begin{equation}
    n^{-1/2}\left(A_n(t\wedge \tau_n')- n\widehat{a}_n(t),H_A(t\wedge \tau_n')- n h_A^n(e^{-t})\right) \tod (Z_A(t\wedge \tau_n'), Z_{HA}(t\wedge \tau_n') ), 
\end{equation}
as $n\to \infty$, where $\widehat{a}_n(t) = 1-\frac 1 n \sum_{d=1}^{\infty}\sum_{\theta=1}^{d} u_n(d,\theta) \beta(d,e^{-t},d-\theta+1)$,  
\begin{equation}
\label{eqn:912pm16may22}
\begin{aligned}
     h_A^n(e^{-t}) = e^{-2t}\frac{2m_n}{n}&-\frac 1 n \sum_{d=0}^{\infty}\sum_{\theta=1}^{\infty}(d-\theta+1)u_n(d,\theta) \beta(d,e^{-t},d-\theta+1)\\
     &-\sum_{d=0}^{\infty}\sum_{\theta=1}^{\infty}\sum_{\ell=d-\theta+2}^{d}u_n(d,\theta) \beta(d,e^{-t},\ell),
     \end{aligned}
\end{equation}
and then
\begin{equation}
\label{eqn:911pm16may22}
    Z_{HA}(t):= Z_W(t)-\sum_{d=0}^{\infty}\sum_{\theta=1}^{\infty}(d-\theta+1)\widetilde{Z}_{d,\theta,d-\theta+1} -  \sum_{d=0}^{\infty}\sum_{\theta=1}^{\infty}\sum_{ \ell= d-\theta+2}^d \widetilde{Z}_{d,\theta,\ell}(t).
\end{equation}

We can ignore the contribution of the tail in the centered and scaled version of \eqref{eqn:336pm16may22} using an almost identical argument as Lemma \ref{lem:418pm27feb21}. Similarly, we can also ensure that the limit $Z_{HA}(t)$ is continuous. This completes the proof of Proposition~\ref{prop:446pm17may22}. 
\end{proof}

%

\subsection*{Acknowledgment}We thank the referee for a very detailed report which significantly improved the quality of this article. 
E. B. is partially supported by the National Science Foundation under grant DMS-2106556 and by the Susan M. Smith chair. S. C. is partially supported by the Netherlands Organisation for Scientific Research (NWO) through Gravitation-grant NETWORKS-024.002.003.
\bibliographystyle{plain}
\bibliography{CLT_bootstrap.bib}

\end{document}